\documentclass[reqno]{amsart}
\usepackage{latexsym}
\usepackage{amssymb,amsthm,amsmath}
\usepackage[dvips]{graphicx}

\theoremstyle{plain}
\newtheorem{theorem}{Theorem}
\newtheorem{lemma}[theorem]{Lemma}

\newtheorem{proposition}[theorem]{Proposition}

\theoremstyle{definition}
\newtheorem{definition}[theorem]{Definition}
\theoremstyle{remark}
\newtheorem{remark}[theorem]{Remark}

\def\d#1{{#1\kern-0.4em\char"16\kern-0.1em}}
\def\D#1{{\raise0.2ex\hbox{-}\kern-0.4em #1}}
\newcounter{zd}

\newcounter{zdr}[subsection]

\newcommand{\eps}{\varepsilon}

\def\Ldws#1{{{\rm L}^{2}_{{\rm w^\ast}}(#1)}}

\def\ve{\varepsilon}

\def\F{{\cal F}}
\def\pa{\partial}
\def\ve{\varepsilon}

\def\cal{\mathcal}
\let\mib=\boldsymbol
\def\ae#1{\;(\hbox{\rm a.e. } #1)}
\def\aps#1{\big|{#1}\big|}

\def\Apslr#1{\left|{#1}\right|}

\def\R{{\bf R}}
\def\Z{{\bf Z}}
\def\N{{\bf N}}

\def\Cbc#1{{{\rm C}^{\infty}_{c}(#1)}}
\def\CB#1{{{\rm C}_b(#1)}}

\def\Cc#1{{{\rm C}_{{\rm c}}(#1)}}
\def\Cnl#1{{{\rm C}_0(#1)}}

\def\Cp#1{{{\rm C}(#1)}}
\def\pC#1#2{{{\rm C}^{#1}(#2)}}
\def\dv{{\sf div\thinspace}}
\def\dscon{\relbar\joinrel\rightharpoonup}

\def\dstr{\longrightarrow}
\def\Dup#1#2{\langle#1,#2\rangle}
\def\Dupp#1#2{\Bigl\langle#1,#2\Bigr\rangle}
\def\eps{\varepsilon}

\def\Lb#1{{{\rm L}^\infty(#1)}}

\def\Ld#1{{{\rm L}^{2}(#1)}}
\def\Ldc#1{{{\rm L}^{2}_{{\rm c}}(#1)}}
\def\Ldl#1{{{\rm L}^{2}_{{\rm loc}}(#1)}}

\def\Ll#1#2{{{\rm L}^{#1}_{{\rm loc}}(#2)}}

\def\LLd{{{\rm L}^2}}
\def\LLb{{{\rm L}^{\infty}}}
\def\LLp#1{{{\rm L}^{#1}}}
\def\Lj#1{{{\rm L}^{1}(#1)}}
\def\pL#1#2{{{\rm L}^{#1}(#2)}}
\def\pLc#1#2{{{\rm L}^{#1}_{{\rm c}}(#2)}}

\def\malpha{{\mib \alpha}}
\def\mbeta{{\mib \beta}}
\def\meta{{\mib \eta}}
\def\mkappa{{\mib \kappa}}
\def\mmu{{\mib \mu}}
\def\mzeta{{\mib \zeta}}
\def\msnop{{\bf p}}
\def\mq{{\bf q}}
\def\mx{{\bf x}}
\def\mxi{{\mib \xi}}
\def\my{{\bf y}}

\def\mB{{\bf B}}
\def\Nor#1{\| #1 \|}    %   norma, ne zaboravi na $
\def\nor#1#2{{\| #1 \|}_{#2}}   %   norma u #2, ne zaboravi na $
\def\oi#1#2{\langle#1,#2\rangle}
\def\Pd{{\rm P}}
\def\ph{\varphi}
\def\Postoji#1{\left(\exists\,#1\right)}
\def\povlaci{\quad\Longrightarrow\quad}
\def\povrhsk#1{\smash{
        \mathop{\;\Dscon\;}\limits^{#1}}}
\def\Rd{{{\bf R}^{d}}}
\def\Rm{{{\bf R}^{m}}}
\def\vnul{{\sf 0}}
\def\Rpl{{{\bf R}^{+}}}

\def\str{\longrightarrow}
\def\Svaki#1{\left(\forall\,#1\right)}
\def\ve{{\sf e}}
\def\vf{{\sf f}}
\def\vu{{\sf u}}
\def\W#1#2#3{{{\rm W}^{#1,#2}(#3)}}
\def\Wc#1#2#3{{{\rm W}_{{\rm c}}^{#1,#2}(#3)}}
\def\Wl#1#2#3{{{\rm W}_{{\rm loc}}^{#1,#2}(#3)}}
\def\WW#1#2{{{\rm W}^{#1,#2}}}
\begin{document}
\title[]{Velocity averaging -- a general framework}
\author{ Martin Lazar}
\address{ Martin Lazar, University of Dubrovnik, Department of Electrical Engineering, 20000  Dubrovnik, Croatia}
 \email{  martin.lazar@unidu.hr} \date{}
 \author{ Darko Mitrovi\'c}
\address{ Darko Mitrovi\'c, University of Montenegro, Faculty of Mathematics, Cetinjski put bb, 81000 Podgorica, Montenegro}
 \email{  matematika@t-com.me}
 \date{}

\begin{abstract}
We prove that the sequence of averaged quantities
$\int_{\R^m}u_n(\mx,\msnop)$ $\rho(\msnop)d\msnop$,  is strongly precompact in
$\Ldl\Rd$, where $\rho\in \Ldc{\R^m}$, and $u_n\in
 \Ld{\R^m; \pL s\Rd}$, $s\geq 2$, are weak solutions to differential
operator equations with variable coefficients. In particular, this
includes differential operators of hyperbolic, parabolic or
ultraparabolic type, but also fractional differential operators. If
$s>2$ then the coefficients can be discontinuous with respect to the
space variable $\mx\in \R^d$, otherwise, the coefficients are
continuous functions. In order to obtain the result we prove a
representation theorem for an extension of the H-measures.
\end{abstract}

\subjclass{35K70, 42B37, 46G10}

\keywords{velocity averaging, generalised H-measures, ultraparabolic equations, discontinuous coefficients, entropy solutions}

\maketitle

\section{Introduction}

The main subject of the paper is the following sequence of
equations:

\begin{equation}
\label{main-sys}
\begin{split}
{\cal P}u_n(\mx,\msnop)&=\sum\limits_{k=1}^d\pa^{\alpha_k}_{x_k}
\left(a_k(\mx,\msnop) u_n(\mx,\msnop)\right)=\pa^\mkappa_\msnop G_n(\mx,\msnop),
\end{split}
\end{equation} where $u_n$ are weak solutions to
\eqref{main-sys} such that $u_n\dscon 0$ in $\Ld{\Rm; \pL s{\Rd}}$, $s\geq 2$, while:

\begin{itemize}
\item [\bf a)] $\alpha_k> 0$ are real numbers and $\pa^{\alpha_k}_{x_k}$ are (the Fourier)
multiplier operators with the symbols $(2 \pi i\xi_k)^{\alpha_k}$,
$i^{\alpha_k}:=e^{\frac{i \alpha_k \pi}{2}}$, $k=1,\dots,d$;

\item [\bf b)]

$$
a_k\in \begin{cases} \Ld{\Rm;\CB{\R^d}}, & s=2\\
\Ld{\Rm; \pL r{\Rd}}, &  2/s+1/r=1, \quad s>2,
\end{cases}
$$
where $\CB\Rd$ stands for a space of continuous and bounded functions;

\item [\bf c)]
$\pa_\msnop^\mkappa=\pa^{\kappa_1}_{p_1}\dots \pa^{\kappa_m}_{p_m}$ for a
multi-index $\mkappa=(\kappa_1,\dots,\kappa_m)\in \N^m$, and
\begin{equation*}
%\label{kinstar}
G_n\to 0 \ \ {\rm in} \ \
\Ld{\R^m;\W{-\malpha}{s'}\Rd}, \ \ \malpha=(\alpha_1,\dots,\alpha_d),
\end{equation*} where $\W{-\malpha}{s'}\Rd$
is a dual of $\W{\malpha}{s}\Rd=\{u\in \pL s\Rd: \pa_k^{\alpha_k}u\in \pL s\Rd,  k=1,\dots,d \}$  (for details on anisotropic Sobolev spaces see e.g.~\cite{VP}).

\end{itemize}

Equations \eqref{main-sys} involve the space variable $\mx\in \R^d$,
with respect to which we have derivatives of solutions $(u_n)$, and
the variable $\msnop\in \R^m$, which is usually called the velocity
variable.

Notice that if $\alpha_k\in \N$ then equation \eqref{main-sys} is a
standard partial differential equation. In particular, for
$\alpha_1=\dots=\alpha_d=1$ one gets a transport equation
(considered in e.g. \cite{Ger, Per}; see more detailed discussion
below). In general, we have a linear fractional differential
equation.

First, we introduce a definition of a weak solution to
\eqref{main-sys}. Assume for the moment that the sub-index $n$ is
removed in \eqref{main-sys}.

\begin{definition}
\label{weaksol} We say that a function $u\in \Ld{\Rm; \pL s{\Rd}}$ is a
weak solution to \eqref{main-sys} if for every $g\in
\Wc{|\mkappa|}2{\R^m;\W{\malpha}{s}\Rd}$ it holds
\begin{equation}
\begin{split}
\label{defws}
&\int\limits_{\R^{m+d}}\!\sum\limits_{k=1}^d
a_k(\mx,\msnop) u(\mx,\msnop)\overline{(-{\pa}_{x_k})^{\alpha_k}(g(\mx,\msnop))}d\mx
d\msnop
= (-1)^{|\mkappa|}\!\!\int\limits_{\R^{m}}\!\Dupp{ G(\cdot,\msnop)}
{\overline{\pa^{\mkappa}_\msnop g(\cdot,\msnop)}} d\msnop\,,
\end{split}
\end{equation}
where duality on  $\W{\malpha}{s}\Rd$ is considered.
\end{definition}

In this paper, we are concerned with compactness properties of
sequence $(u_n)$. It is not difficult to find examples of
equations of type \eqref{main-sys} such that the sequence $(u_n)$
does not converge strongly in $\Ll s{\R^m\times\R^d}$ for any
$s\geq 1$. Indeed, a trivial example $u_n=\sin n\msnop$ solving
\eqref{main-sys} with coefficients being independent of $\mx\in \R^d$
and $\alpha_k\in \N$, $k=1,\dots,d$, does not converge strongly in
${\rm L}^s_{loc}$ for any $s\geq 1$.

Still, from the viewpoint of applications, it is almost always enough
to analyse the sequence $(u_n)$ averaged with respect to the
velocity variable $(\int_{\R^m}\rho(\msnop)u_n(\mx,\msnop)d\msnop)$, $\rho\in
\Cc{\R^m}$ (see e.g. famous papers \cite{8, LPT}) which, as
firstly noticed by Agoshkov \cite{Ago} in the homogeneous hyperbolic
case, can be strongly precompact in $\Ll s\Rd$ for an
appropriate $s\geq 1$ even when the sequence $(u_n(\mx,\msnop))$ is not. Such
results are usually called velocity averaging lemmas.

After Agoshkov's paper, the investigations in this directions
continued rather intensively. Still, in most of the previous works
on the subject, the symbol $P(i\mxi, \mx, \msnop)$ of the
differential operator ${\cal P}$ was of the first order and
independent of $\mx\in \R^d$. Thus the corresponding equation
describes a transport process occurring in a homogeneous medium. On
the other hand, most of natural phenomena take place in
heterogeneous media (flow in heterogeneous porous media,
sedimentation processes, blood flow, gas flow in a variable duct,
etc). However, it appears that it is much more complicated to work
on heterogeneous transport equations than on homogeneous ones.

This fact could be explained by the following simple observation.
Assume that the coefficients in \eqref{main-sys} do not depend on
$\mx\in \R^d$. If we apply the Fourier transform in $\mx\in \R^{d}$ on
equation \eqref{main-sys}, at least informally, we can separate
solutions $(u_n)$ and the known coefficients. To be more precise, let us
consider the sequence of homogeneous transport equations from
\cite{Per}:

\begin{equation}
\label{perth} \pa_t u_n + a(\msnop)\cdot \nabla_\mx
u_n=\sum\limits_{j=1}^d\pa_{x_j} \pa_\msnop^\mkappa g^n_j, \ \ (t, \mx, \msnop)\in
\R^+\times\R^d\times\R^d,
\end{equation} where, for some $s>1$, $u_n\rightharpoonup 0$ weakly in $\pL s{\R^{d+1}}$, while $g^n_j\to 0$ strongly in
$\Ll s{\R^+\times \R^d\times \R^d}$,  $j=1,\dots,d$. The function
 $a:\R^d\to \R^d$ is   continuous.

By finding the Fourier transform of \eqref{perth} with respect to
$(t,\mx)\in \R^+\times \R^d$ (denoted by $\hat{}$ below), we conclude
from the above
\begin{equation*}
(\tau+a(\msnop)\cdot \mxi)\hat{u}=\sum\limits_{j=1}^d \xi_j \pa_\msnop^\mkappa
\hat{g}_j,
\end{equation*} and from here, for any $\beta>0$,
\begin{equation*}
\hat{u}=\frac{\beta^2|\mxi|^2 \hat{u}+\sum\limits_{j=1}^d
 (\tau+a(\msnop)\cdot \mxi) \xi_j \pa_\msnop^\mkappa \hat{g}_j}{(\tau+a(\msnop)\cdot
\mxi)^2+\beta^2 |\mxi|^2}.
\end{equation*}
As the term containing $\hat{u}$ on the right-hand side can be
controlled by constant $\beta$, it
was proved in \cite{Per} that the sequence of averaged quantities
$(\int_{\R^m} \rho(\msnop)u_n(t,\mx,\msnop)d\msnop)$, $\rho\in \pL{s'}{\R^m}$,
$1/s+1/s'=1$, converges to zero strongly in $\pL s{\R^{d+1}}$.

Actually, such framework is probably the main approach used on the
subject \cite{15, 20, Gol, Tao}. Other approaches include the use
of wavelet decomposition \cite{16}, ``real-space methods" in time
\cite{5, 52} and ``real-space methods" in space using the Radon
transform \cite{12, 54}, $X$-transform \cite{31}, duality based
dispersion estimates \cite{26}, etc.

In the heterogeneous case, the method applied on \eqref{perth} is
not at our disposal (since $\hat u$ can not be separated). Probably the only possible way to
tackle the heterogeneous velocity averaging problem is through a
variant of defect measures \cite{Ant2, Ger, MI, Pan_IHP, Tar}. In
\cite[Theorem 2.5]{Ger} the concrete application of defect measures
on the averaging lemmas can be found. The result from \cite{Ger}
claims that the sequence of solutions $(u_n)$ of equations
\eqref{main-sys} satisfying conditions a)--c) with
$\alpha_1=\alpha_2=\dots=\alpha_d\in \N$ and $s=2$,
is such that the sequence of averaged quantities $(\int
u_n(\mx,\msnop)\rho(\msnop)d\msnop)$ strongly converges to zero in $\Ld\Rd$.

In this paper, we shall generalise Gerard's result on a wider class
of equations, and we shall allow the coefficients to be
discontinuous if the solutions $u_n$ are from
$\Ld{\R^m; \pL s\Rd}$ for $s>2$ (which is the situation in a
numerous applications; e.g. \cite{BKT, Mitr, pan_jms}). We remark again
that the result from \cite{Per} can be applied only in the case of
homogeneous transport equations, but it is optimal in the sense that
a sequence of solutions can belong to $\pL s{\R^{d+1}\times\R^d}$ for
any $s>1$ (in the current contribution, we must have $s\geq 2$).

Let us now describe defect measures that we are going to use. A
defect measure is an object describing loss of compactness of a
family of functions. Originally, the notion of the defect measure
was systematically studied for sequences satisfying elliptic
estimates by P.L.Lions \cite{Liocc}. Since elliptic estimates
automatically eliminate oscillations, the defect measures used in
\cite{Liocc} were not appropriate enough for studying loss of
compactness caused by oscillations,  which typically appear in
the case of e.g. hyperbolic problems.

In order to control oscillations, a natural idea was to introduce an
object which distinguishes oscillations of different frequencies.
The idea was formalised by P.~Gerard \cite{Ger} and independently by
L.~Tartar \cite{Tar}. P.~Gerard named the appropriate defect measure
as the microlocal defect measure (mdm in the sequel), while
L.~Tartar used the term H-measure. Let us
recall Tartar's theorem introducing the H-measures.

\begin{theorem}\cite{Tar}
\label{H-meas}
If $(\vu_n)=((u_n^1,\dots, u_n^r))$ is a sequence in $\Ld{\R^d;\R^r}$
such that $\vu_n\rightharpoonup 0$ in $\Ld{\R^d;\R^r}$, then there
exists its subsequence $(\vu_{n'})$ and a positive definite matrix of
complex Radon measures $\mmu=\{\mu^{ij}\}_{i,j=1,\dots,r}$ on
$\R^d\times S^{d-1}$ such that for all $\varphi_1,\varphi_2\in
C_0(\R^d)$ and $\psi\in C(S^{d-1})$

\begin{equation}
\label{basic1}
\begin{split}
\lim\limits_{n'\to \infty}\int_{\R^d}&
{\cal A}_\psi(\varphi_1 u^i_{n'})(\mx)\overline{(\varphi_2
u^j_{n'})(\mx)}d\mx
=\langle\mu^{ij},\varphi_1\overline{\varphi_2}\psi
\rangle\\&= \int_{\R^d\times
S^{d-1}}\varphi_1(\mx)\overline{\varphi_2(\mx)}\psi(\mxi)d\mu^{ij}(\mx,\mxi),
\ \ i,j=1,\dots,r,
\end{split}
\end{equation}
where ${\cal A}_\psi$ is a multiplier operator with symbol $\psi\in C(S^{d-1})$ (see Definition \ref{multiplier}).
\end{theorem}

G\' erard's approach generalises the above results to $\LLd$-sequences
taking values in an infinite-dimensional,  separable Hilbert space
$H$. In the case when $H=\Ld{\R^m}$, G\' erard's mdm is an object
belonging to ${\cal M}_+(S^\star \Omega, {\cal L}^1(H))$, i.e. to
the space of non-negative Radon measures on the cospherical bundle $S^\star
\Omega$ (the set $\Omega\times S^{d-1}$ endowed with the natural
structure of manifold) with values in the space of trace class
operators on $H$. It is important to mention an extension of
H-measures in the case of sequences which (basically) have the
form $({\rm sgn}(\lambda-u_n(\mx)))$, $u_n\in \Lb\Rd$, given by Panov
\cite{pan_ms}. There, it was proved that
 for almost every $\lambda_1,\lambda_2\in \R$ there exists
a measure $\mu^{\lambda_1\lambda_2}$  defined by \eqref{basic1} for
$u^{\lambda_i}(\mx)={\rm sgn}(\lambda_i-u_n(\mx))$, $i=1,2$.  This
notion appeared to be very useful, and it was successfully applied
in many recent papers \cite{AM, AMP, HKM, MI, Pan_IHP, pan_arma,
sazh}. Here, we extend Panov's results to sequences belonging to
$\Ld{\R^m; \pL s\Rd}$, $s\geq 2$.

Moreover, our result represents a generalisation of the original
H-measures from two aspects. First, test functions (in applications
these are given by coefficients entering equations of interest) in
our case can be more general, even discontinuous with respect to the
space variable. Second, our generalisation of the H-measures is
constructed for use on a large class of equations (unlike original
H-measures \cite{Ger, Tar} which were adapted only for hyperbolic
type problems).

In a view of the last observation, remark that parabolic \cite{Ant1,
Ant2}, and ultra-parabolic \cite{Pan_IHP} variants of the
H-measures, and finally the H-measures adapted to large class of
manifolds \cite{MI} were introduced. The last one is the main tool
used in this paper. Its description, as well as  the introduction to
the main result (Theorem \ref{main-result}) is given in the next
section.

In Section 3 we shall further develop the H-measure concept, which will be used
in Section 4 for proving the precompactness property of a sequence of
solutions to \eqref{main-sys}. The proof is based on  a special (trivial) form of the variant H-measure corresponding to the
sequence $(u_n)$.

In Section 5 we shall apply our result on ultra-parabolic equations with discontinuous flux under different assumptions on coefficients than the ones from \cite{pan_jms} (which is the most up-to-date result and which comprises the results from \cite{pan_arma}).

\section{Statement of the main result}

To formulate the main result of the paper, we need to introduce the
variant of H-measures that we are going to use. First, we need
some auxiliary notions.

\begin{definition}
\label{multiplier} A multiplier operator ${\cal A}_\psi:\Ld\Rd\to
\Ld\Rd$ associated to a bounded function $\psi\in \CB\Rd$ (see
e.g. \cite{ste}), is a mapping  by
$$
{\cal A}_\psi(u)=\bar{\F}(\psi \hat{u}),
$$where $\hat{u}(\mxi)=\F(u)(\mxi)=\int_{\R^d}e^{-2\pi i \mx\cdot
\mxi}u(x)dx$ is the Fourier transform while $\bar{\F}$ (or $^\vee$) is
the inverse Fourier transform.

If the multiplier operator ${\cal
A}_\psi$ satisfies
$$
\nor{{\cal A}_\psi (u)}{\LLp p} \leq C \nor{ u}{\LLp p}, \qquad u\in \pL p\Rd\cap\Ld\Rd,
$$
where $C$ is a positive constant, then
the function $\psi$ is called the $\LLp p$-multiplier.
\end{definition}

Let $l$ be a minimal number such that $l \alpha_k > d$ for each $k$.
We shall introduce the following manifolds, denoted by $\Pd$ and
determined by the order of the derivatives from \eqref{main-sys}:
\begin{equation}
\label{kin2} \Pd=\{\mxi\in \R^d: \; \sum\limits_{k=1}^d
|\xi_k|^{l\alpha_k}=1 \}.
\end{equation}
On such manifolds, which are
smooth according to the choice of $l$,  we shall define the necessary
H-measures. Remark that it can seem more natural to take $\Pd=\{\mxi\in
\R^d: \; \sum\limits_{k=1}^d |\xi_k|^{\alpha_k}=1 \}$ but the latter
manifold is not smooth enough. Namely, we shall need the following
corollary of the Marzinkiewicz multiplier theorem \cite[Theorem IV.6.6']{ste}:
\begin{lemma}
\label{m1} Suppose that $\psi\in \pC{d}{\R^d\backslash \{\vnul\}}$ is such that for some constant
$C>0$ it holds
\begin{equation}
\label{c-mar} |\mxi^\mbeta \partial^\mbeta \psi(\mxi)|\leq C, \ \ \mxi\in
\R^d\backslash \{\vnul\}
\end{equation} for every multi-index
$\mbeta=(\beta_1,\dots,\beta_d)\in\Z_+^d$ such that
$|\mbeta|=\beta_1+\beta_2+\dots+\beta_d \leq d$. Then, the
function $\psi$ is an $\LLp p$-multiplier for $p\in \oi 1\infty$, and the operator norm of ${\cal A}_\psi$ depends only on $C, p$ and $d$.
\end{lemma}

The next lemma is an easy corollary of Lemma \ref{m1}. First, denote
by
$$
\pi_{\Pd}(\mxi)=\left(\frac{\xi_1}{\left(\xi_1^{l\alpha_1}+\dots+\xi_d^{l\alpha_d}
\right)^{1/l
\alpha_1}},\dots,\frac{\xi_d}{\left(\xi_1^{l\alpha_1}+\dots+\xi_d^{l\alpha_d}
\right)^{1/l \alpha_d}}\right), \ \ \mxi\in \R^d\backslash\{0\},
$$
a projection of $\R^d\backslash \{\vnul\}$ on $\Pd$. The following result holds.

\begin{lemma}
\label{marz} For any $\psi\in \pC{d}\Pd$, the composition $\psi\circ
\pi_{\Pd}$ is an $\LLp p$-multiplier, $p\in \oi 1\infty$,  and the norm of the corresponding multiplier operator  depends on $\nor \psi {\pC{d}\Pd},\, p$ and $d$.
\end{lemma}
\begin{proof}
Due to the Fa\' a di Bruno formula, it is enough to prove that the conditions of Lemma \ref{m1} are
satisfied for
$\pi_k(\mxi)=\frac{\xi_k}{\left(\xi_1^{l\alpha_1}+\dots+\xi_d^{l\alpha_d}
\right)^{1/l \alpha_k}}$, $k=1,\dots,d$.

The statement  will be proved by the induction argument.

\begin{itemize}

\item $n=1$

In this case, we compute
$$
\pa_{j}\pi_k(\mxi)=\begin{cases}
-\frac{\alpha_j}{\alpha_k}\frac{1}{\xi_j}\pi_k(\mxi)\pi_j^{l \alpha_j}(\mxi), & j\neq k\\
-\frac{1}{\xi_k} \pi_k(\mxi) \left(1-\pi_k^{l \alpha_k}(\mxi)
\right), & j=k.
\end{cases}
$$ and it obviously holds $|\xi_j \pa_{j} \pi_k(\mxi)| \leq C$.

\item $n=m$

Our inductive hypothesis is
\begin{equation}
\label{ih} \partial^\mbeta
\pi_k(\mxi)=\frac{1}{\mxi^\mbeta}P_\mbeta(\pi_1(\mxi),\dots,\pi_d(\mxi)),
\ \ |\mbeta|=m,
\end{equation} for a polynomial $P_\mbeta$.

\item $n=m+1$

To prove that \eqref{ih} holds for $|\mbeta|=m+1$ it is enough to
notice that $\mbeta=\ve_j+\mbeta'$, where
$|\mbeta'|=m$, and to notice
$$
\partial^\mbeta\pi_k(\mxi)=\partial_j \partial^{\mbeta'}
\pi_k(\mxi)=\partial_j\left(\frac{1}{\mxi^{\mbeta'}}P_{\mbeta'}(\pi_1(\mxi),\dots,\pi_d(\mxi))\right)
$$ and from here, repeating the procedure from the case $n=1$, we conclude that \eqref{ih} holds for $n=m+1$.

\end{itemize}

From here, \eqref{c-mar} immediately follows for $\pi_k$ and
consequently for $\psi\circ\pi_\Pd$.

\end{proof}

To proceed, we introduce a family of  curves
\begin{equation}
\label{kin3} \eta_k=\xi_k t^{1/l \alpha_k }, \ t\in\Rpl, \
\end{equation}
 by points $\mxi=(\xi_1,\dots,\xi_d)\in \Pd$.
They are disjoint and fibrate entire space $\R^d$. They
play the same role as the rays $\mxi/|\mxi|$ in the definition of
the H-measures. Moreover, we see that the curves \eqref{kin3}
respect the scaling given by the differential operator from
\eqref{main-sys}. Indeed, if we have the classical situation
$\alpha_k=1$, $k=1,\dots,d$, then curves \eqref{kin3} are rays and
we can use the classical H-measures \cite{Ger, Tar}.

%
%\begin{figure}[htp]
%\begin{center}
%  \includegraphics[width=3in]{annIHP.ps}\\
%  \caption{The admissible manifold is denoted by $\Pd$. Fibres are dashed.
%  Notice that a fibre must not intersect $\Pd$ twice.}
%  \label{sl1}
%  \end{center}
%\end{figure}

The following theorem is essentially proved in \cite{MI}, but here we provide its more elegant proof based on the ideas of L. Tartar.

\begin{theorem}
\label{tbasic1}  For fixed $\alpha_k>0$, $k=1,\dots,d$, denote by
$\Pd$ the manifold given by \eqref{kin2}, and by $\pi_{\Pd}:\R^d\to
\Pd$ projection on the manifold $\Pd$ along the fibres \eqref{kin3}.
If $(\vu_n)=((u_n^1,\dots, u_n^r))$ is a sequence in $\Ld{\R^d;\R^r}$
such that $\vu_n {\buildrel {\rm L}^{2}\over\dscon} \;0$ (weakly), then there
exists its subsequence $(u_{n'})$ and a positive definite matrix of
complex Radon measures $\mmu=\{\mu^{ij}\}_{i,j=1,\dots,d}$ from ${\cal M}_{b}(\Rd\times \Pd)$ such that for all $\varphi_1,\varphi_2\in \Cnl\Rd$
and ${\psi}\in \Cp\Pd$
\begin{equation}
\label{basic1_new}
\begin{split}
\lim\limits_{n'\to \infty}\int_{\R^d}
&{\cal A}_{\psi_\Pd}(\varphi_1
u^i_{n'})(\mx)\overline{(\varphi_2
u^j_{n'})(\mx)}dx=\langle\mu^{ij},\varphi_1\overline{\varphi_2}\psi
\rangle\\&= \int_{\R^d\times
\Pd}\varphi_1(\mx)\overline{\varphi_2(\mx)}\psi(\mxi)d\mu^{ij}(\mx,\mxi), \
\ (\mx,\mxi)\in \R^d\times \Pd,
\end{split}
\end{equation}
where  ${\cal A}_{{\psi_\Pd}}$ is a multiplier operator with the symbol ${\psi_\Pd}:= \psi\circ \pi_\Pd$.

The measure $\mmu$ we call the ${\rm H}_\Pd$-measure corresponding to the
sequence $(\vu_n)$.
\end{theorem}
\begin{proof}
First, we shall prove that the fibration \eqref{kin3} satisfies
conditions of the variant of the first commutation lemma \cite[Lemma
28.2]{tar_book}. More precisely, we shall prove that any symbol
$\psi$  on the manifold $\Pd$ satisfies

\begin{equation}
\label{fcl-o}
\begin{split}
&\Svaki{ r, \eps \in \Rpl} \ \ \Postoji{M\in \Rpl} \ \ \\
&|\meta_1-\meta_2|\leq r, \; |\meta_1|, |\meta_2| > M \ \ \implies \ \
|\psi(\pi_{\Pd}(\meta_1))-\psi(\pi_{\Pd}(\meta_2))|\leq \eps,
\end{split}
\end{equation} where $\pi_{\Pd}$ is the projection on the manifold
$\Pd$ along the fibres \eqref{kin3}.

As  $\psi$ is an uniformly continuous on $\Pd$, it is enough to show that for fixed $r$ and $\eps$, the difference
$|\pi_{\Pd}(\meta_1)-\pi_{\Pd}(\meta_2)|$ is arbitrary small for $M$ large enough.
According to the mean value theorem
$$
|\pi_{\Pd}(\meta_1)-\pi_{\Pd}(\meta_2)|\leq |\nabla\pi_{\Pd}(\mzeta)| |\meta_1-\meta_2|,
$$
where $\mzeta = \vartheta \meta_1 + (1- \vartheta) \meta_2$ for some $\vartheta\in \oi 01$, and the statement follows as $\nabla \pi_\Pd (\meta)$ tends to zero when $|\meta|$ approaches infinity.

Now, we can use \cite[Lemma 28.2]{tar_book} to conclude that the
mappings
$$
(\varphi_1\overline{\varphi_2},\psi)\mapsto \lim\limits_{n'\to
\infty}\int_{\R^d}{\cal
A}_{\psi_\Pd}(\varphi_1 u^i_{n'})(\mx)\overline{(\varphi_2 u^j_{n'})(\mx)}d\mx, \ \ i,j=1,\dots,d,
$$ form a positive definite matrix of bilinear functionals on $\Cnl\Rd\times \Cp\Pd$. According to the Schwartz kernel theorem, the
functionals can be extended to a continuous linear functionals on
${\cal D}(\R^d\times \Pd)$. Due to its non-negative definiteness,
the Schwartz theorem on non-negative distributions \cite[Theorem
I.V]{Sw} provides its extension on the Radon measures.

\end{proof}

Notice that, using the Plancherel theorem, \eqref{basic1_new} can be
conveniently rewritten via the Fourier transform as follows:
\begin{equation*}
%\label{basic1_new'}
\begin{split}
\lim\limits_{n'\to \infty}\int_{\R^d}&\F(\varphi_1
u^i_{n'})(\mxi)\overline{\F(\varphi_2 u^j_{n'})(\mxi)}\, {\psi\circ\pi_\Pd}(\mxi)
d\mxi=\langle\mu^{ij},\varphi_1\overline{\varphi_2}\psi \rangle\\&=
\int_{\R^d\times
\Pd}\varphi_1(\mx)\overline{\varphi_2(\mx)}\psi(\mxi)d\mu^{ij}(\mx,\mxi).
\end{split}
\end{equation*}

Now, we can formulate the main theorem of the paper.

\begin{theorem}
\label{main-result}
 Assume that $u_n\dscon 0$  weakly in
$\Ld{\R^m;\pL s\Rd}\cap \Ld{\R^{m+d}}$, $s\geq 2$, where $u_n$ represent weak
solutions to \eqref{main-sys} in the sense of Definition
\ref{weaksol}.

Furthermore,  for $s=2$ we assume that for every $(\mx,\mxi)\in
\R^d\times\Pd$
\begin{equation}
\label{kingnl} A(\mx,\mxi,\msnop):=\sum\limits_{k=1}^d
a_k(\mx,\msnop)(2 \pi i\xi_k)^{\alpha_k}\not= 0 \quad\ae{ \msnop \in \Rm}\,.
\end{equation}
If $s>2$, the last assumption is reduced to almost every $\mx\in \R^d$ and every
$\mxi\in \Pd$.

Then, for any $\rho\in \pLc{2}{\R^m}$,
$$
\int_{\R^m}u_n(\mx,\msnop)\rho(\msnop)d\msnop
\str 0 \ \ \text{ strongly in $\Ldl\Rd$}.
$$

\end{theorem}

% Assume that for  every $(\mx,\mxi)\in \R^d\times\Pd$
%\begin{equation}
%\label{kingnl} A(\mx,\mxi,\msnop):=\sum\limits_{k=1}^d
%a_k(\mx,\msnop)(2 \pi i\xi_k)^{\alpha_k}\not= 0 \quad\ae{ \msnop \in \Rm}\,.
%\end{equation}
%In the case $s>2$ the last condition can be relaxed by assuming that \ref{kingnl} is valid only for almost every $\mx\in \Rd$.

Before we continue, remark that the conditions of the theorem   can be relaxed
by assuming that $(u_n)$ is merely bounded in $\Ld{\R^m;\pL s\Rd}$, while $(G_n)$  strongly
precompact in $\Ld{\R^m;\W{-\malpha}{s'}\Rd}$.
In that case there exists a
subsequence $(u_{n'})$  such that for any $\rho\in
\pLc{2}{\R^m}$ the sequence $(\int_{\R^m}\rho(\msnop)u_{n'}(\mx,\msnop)d\msnop)$ converges
toward $\int_{\R^m}\rho(\msnop)u(\mx,\msnop)d\msnop$, where  $u$ denotes the weak limit of $(u_n')$.

\section{Auxiliary results}

In this section, we shall extend Theorem \ref{tbasic1} on sequences
with uncountable indexing. A similar procedure we used in the case
of the parabolic variant H-measures \cite{osijek}, and for the sake of completeness we reproduce some results here. These  will be substantially extended by Proposition
\ref{prop_repr} and Theorem \ref{lfeb518-r} containing   representation results of  ${\rm H}_\Pd$-measures associated to sequences of functions $u_n\in \Ld{\R^m;\pL s\Rd}$, $s>2$, which turn to be crucial for the proof of the main theorem.

Let us take an arbitrary sequence of functions $(u_n)$ in variables
$\mx\in \R^d$ and $\msnop\in\R^m$, weakly converging to zero in $\Ld{\R^m
\times \R^d}$. Introduce a regularising kernel $\omega\in
\Cbc{\R^m}$, where $\omega$ is a
non-negative smooth function with total mass one. For $k\in\N$ denote $ \omega_k(\msnop)=k^m\omega(k \msnop)$
and convolute it with $(u_n(\mx,\msnop))$  in $\msnop$:
\begin{equation*}
u_n^k(\mx,\msnop):=\Bigl(u_n(\mx,\cdot)\ast\omega_k\Bigr)(\msnop)=\int_{\R^m}
u_n(\mx,\my)\omega_k(\msnop-\my)d\my.
\end{equation*}
By the Young inequality functions $u_n^k$ are bounded in $\Ld{\R^{m+d}}$ uniformly with respect to both $k$ and $n$. Meanwhile, for every fixed $k$, sequence  of
functions $u_n^k(\cdot, \msnop)$ is bounded   in $\Ld\Rd$,
uniformly in $\msnop$, and  converges weakly to zero. Furthermore, $u^k_n$ are Lipschitz
continuous as functions from $\R^m$ to $\Ld\Rd$, with an
$n$-independent Lipschitz constant. Having all this in mind, we can
prove the following lemma.

\begin{lemma}
There exists a subsequence $ (u_{n'})$ of the sequence $(u_n)$,
and a family $\{\mu^{\msnop{\mq}}_{k}: \msnop,{\mq}\in\R^m\}$ of ${\rm H}_\Pd$-measures
on $\R^d\times \Pd$ such that for every $k\in \N$, $\varphi_i\in
C_0(\R^d)$, $i=1,2$, and $\psi\in C(\Pd)$:
\begin{equation}
\label{jan2718}
\begin{split}
\lim\limits_{n'}\int\limits_{\Rd}\Bigl({\cal A}_{\psi_\Pd} \,\ph_1
u^{k}_{n'}(\cdot, \msnop)\Bigr)(\mx)
&\,\overline{\ph_2(\mx)u^{k}_{n'}(\mx,\mq )}d\mx\\&=
\!\!\!\!\int\limits_{\Rd\times \Pd}\!\!\!\ph_1(\mx)\overline{\ph_2(\mx)}\psi(\mxi)d\mu^{\msnop\mq}_k(\mx,\mxi).\\
\end{split}
\end{equation}

\end{lemma}
\begin{proof}
According to Theorem \ref{tbasic1}, for fixed $\msnop,\mq\in \R^m$
and $k\in {\bf N}$, there exist a subsequence of $(u_n)$ and
corresponding complex Radon measure $\mu_{k}^{\msnop{\mq}}$ over
$\Rd\times \Pd$ such that \eqref{jan2718} holds. Using the
diagonalisation procedure, we conclude that for a countable dense
subset $D\times {D} \subset \R^m\times \R^{m}$ there exists a
subsequence $(u_{n'})\subset (u_n)$ such that \eqref{jan2718} holds
for every $(\msnop,{\mq})\in D\times {D}$ and every $k\in \N$.

Let us take an arbitrary $k\in {\bf N}$ and  $(\msnop,{\mq})\in
\R^m\times \R^{m}$. Let $(\msnop_m,{\mq}_m)$ be a sequence in
$D\times {D}$ converging to $(\msnop,{\mq})$. The sequence
$(\msnop_m,{\mq}_m)$ defines sequence of Radon measures
$(\mu_k^{\msnop_m {\mq}_m})$, which is bounded in ${\cal M}_b(\Rd\times \Pd)$, due to the bounds of $(u_n^{k})$ in
$\Lb{\R^m; \Ld\Rd}$. Therefore, there exists a complex Radon measure
${\mu}_k^{\msnop\mq}$ such that, along a subsequence,
$\mu_k^{\msnop_m {\mq}_m}\rightharpoonup {\mu}_k^{\msnop\mq}$. Thus for arbitrary test functions $\ph=\ph_1 \bar\ph_2$ and
$\psi$ we have:
\begin{equation}
\label{lema1}
\begin{split}
\int \ph(\mx)\psi(\mxi)\, d\mu_k^{\msnop {\mq}} (\mx, \mxi)
&=\lim_m \int \ph(\mx)\psi(\mxi) \,d\mu_k^{\msnop_m {\mq}_m}(\mx,\mxi)\\
&=\lim_m \lim\limits_{n'} V_{n'}^k (\msnop_m, \mq_m) ,\,
\end{split}
\end{equation}
where $V_n^k$ denotes the function  by
\begin{equation}
\label{Vnk} V_n^k (\msnop, \mq):= \int\limits_{\Rd}\Bigl({\cal
A}_{\psi_\Pd} \,\ph_1 u^{k}_{n}(\cdot, \msnop)\Bigr)(\mx)
\,\overline{\ph_2(\mx)u^{k}_{n}(\mx,\mq )} d\mx\,.
\end{equation}
On the other hand
\begin{equation*}
\begin{split}
V_{n'}^k (\msnop_m, \mq_m) - V_{n'}^k (\msnop, \mq)
&=V_{n'}^k (\msnop_m, \mq_m) - V_{n'}^k (\msnop, \mq_m) +  V_{n'}^k (\msnop, \mq_m) -V_{n'}^k (\msnop, \mq)\\
&\leq C(k)\Bigl(|\msnop_{m}-\msnop|_{\R^{m}}+
|{\mq}_{m}-{\mq}|_{\R^{m}}\Bigr),\,
\end{split}
\end{equation*}
where on the last step we combined the Cauchy-Schwartz inequality,
boundedness of the multiplier ${\cal A}_{\psi_\Pd}$ on $\Ld\Rd$, and the
Lipschitz continuity of the functions $u_n^k$. The constant
$C(k)$ appearing above is independent of $n'$, and we can exchange
limits in \eqref{lema1}. This actually means that the functional
$\mu^{\msnop \mq}_k$ does not depend on the defining subsequence (i.e. it is
well  for every $\msnop,\mq\in \R^m$), which completes the proof.
\end{proof}

Using the previous assertion, we prove the existence of
${\rm H}_\Pd$-measures associated to functions taking values in $\Ld{\R^m}$.
First, we need to recall a few basic notions of $\LLd$ functions taking values in an arbitrary Banach space $E$.

We say that $f: \R^m \to E'$ is weakly $\ast$ measurable if it is
measurable with respect to weak $\ast$ $\sigma(E', E)$ topology. The
dual of $\Ld{\R^m, E}$ corresponds to the Banach space
$\Ldws{\R^{m}; E'}$ of weakly $\ast$ measurable functions $f:
\R^{m}\! \to \!E'$  such that $\int_{\R^{m}} \nor{f(\mx)}{E'}^2 d\mx
<\infty$ (for details see \cite[ p. 606]{Ed}).

By taking $E=\Cnl{\Rd\times \Pd}$, the topological dual
of  $\Ld{\R^{2m}; \Cnl{\Rd\times \Pd}}$ corresponds to the
Banach space $\Ldws{\R^{2m}; {\cal M}_{b}(\Rd\times \Pd)}$ of
weakly $\ast$ measurable functions $\mu: \R^{2m}\! \to \!{\cal
M}_{b}(\Rd\times \Pd)$  such that  $\int_{\R^{2m}}
\Nor{\mu(\msnop, \mq)}^2  d\msnop d\mq <\infty$.

\begin{theorem}
\label{lfeb518} For the subsequence $(u_{n'})\subseteq (u_{n})$
extracted in Lemma 8, the\-re exists a  measure $\mu \in \Ldws{\R^{2m};
{\cal M}_{b}(\Rd\times \Pd)}$ such that for all $v\in
\Ldc{\R^{2m}}$, $\varphi_i\in C_0(\Rd)$, $i=1,2$, and $\psi\in
C(\Pd)$:
\begin{equation}
\label{jan0218}
\begin{split}
\lim\limits_{n'}\int\limits_{\R^{2m}}\int\limits_{\Rd}v(\msnop,{\mq})
&\Bigl({\cal A}_{\psi_\Pd} \,\ph_1 u_{n'}(\cdot, \msnop)\Bigr)(\mx)
\,\overline{\ph_2(\mx) u_{n'}(\mx,\mq )} d\mx d\msnop d\mq\\
&=\int\limits_{\R^{2m}} v(\msnop,{\mq})\,\langle \mu(\msnop,\mq,\cdot,\cdot),\varphi_1\bar{\varphi}_2 \otimes\psi\rangle d\msnop d\mq.\\
\end{split}
\end{equation}
\end{theorem}

\begin{remark}
\label{rem1P} Notice that the new object has inherited the hermitian character of H-measures. Indeed, with the help of Plancherel's theorem, we
can rewrite \eqref{jan0218} as
\begin{equation*}
%\label{jan0218'}
\begin{split}
\lim\limits_{n'}\int\limits_{\R^{2m}}\int\limits_{\Rd}v(\msnop,{\mq})
&\psi(\mxi) \F(\,\ph_1 u_{n'}(\cdot, \msnop))(\mxi)
\,\overline{\F(\ph_2(\cdot) u_{n'}(\cdot,\mq ))(\mxi)} d\mxi d\msnop d\mq\\
&=\int\limits_{\R^{2m}} v(\msnop,{\mq})\,\langle \mu(\msnop,\mq,\cdot,\cdot),\varphi_1\bar{\varphi}_2 \otimes\psi\rangle d\msnop d\mq\,,\\
\end{split}
\end{equation*}
from which it easily follows that
$$
\mu(\msnop,\mq,\cdot,\cdot)=\overline{\mu(\mq,\msnop,\cdot,\cdot)}\,.
$$
Also, notice that we can take $\varphi_1\in \CB\Rd$ (since $\varphi_2\in C_0(\R^d)$).

\end{remark}

\begin{proof}
For fixed test functions $\varphi_{1,2}$ and $\psi$, similarly to
\eqref{Vnk}, we denote
$$
F_k (\msnop, \mq):=\lim\limits_{n'} V_{n'}^k (\msnop,
\mq)=\langle\mu_k^{\msnop\mq},\varphi_1\bar{\varphi}_2\psi \rangle .
$$

Due to the uniform bound of  $u_n^{k}$ in $\Ld{\R^{m+d}}$, the
functions $V_n^k$ belong to the space $\Ld{\R^{2m}}$, with norm depending on $\nor{\varphi_{1,2}}{\LLb}$ and $\nor{\psi}{\LLb}$, but not on $n$ and $k$. Thus the Fatou
lemma asserts the sequence $(F_k)$ is bounded in $\Ld{\R^{2m}}$, as
well.

Furthermore, for a fixed $k$, the sequence $(V_n^k)$ is bounded in
$\Lb{\R^{2m}}$. By taking an arbitrary $v\in \Ldc{\R^{2m}}$, we have
\begin{equation}
\label{mar1918}
\begin{split}
\lim\limits_{k}\int\limits_{\R^{2m}} v(\msnop, \mq) F_k(\msnop, \mq)
d\msnop d\mq
&= \lim\limits_{k}\int\limits_{\R^{2m}} v(\msnop, \mq)\lim\limits_{n'} V_{n'}^k (\msnop, \mq) d\msnop d\mq\\
&= \lim\limits_{k} \lim\limits_{n'}\int\limits_{\R^{2m}}v(\msnop,
\mq) V_{n'}^k (\msnop, \mq) d\msnop d\mq\,
\end{split}
\end{equation}
where on the last step we have used the Lebesgue dominated convergence
theorem.

As the functions $u_n^k$ are uniformly bounded in $\Ld{\R^{m+d}}$,  the sequence of
averaged quantities  $\int\limits_{\R^{2m}}v(\msnop, \mq)
V_{n}^k (\msnop, \mq) d\msnop d\mq$ converges  to
$\int\limits_{\R^{2m}} v(\msnop, \mq) V_n (\msnop, \mq)
d \msnop d\mq$ uniformly with respect to $n$,
where $V_n$ is  similarly to $V_n^k$, with $u_n^{k}$ replaced by $u_n$ in \eqref{Vnk}.

Thus we can exchange  the limits in
\eqref{mar1918} providing
\begin{equation}
\label{jan2618} \lim\limits_{k}\int\limits_{\R^{2m}} v(\msnop, \mq)
F_k(\msnop, \mq) d\msnop d\mq
=\lim\limits_{n'}\int\limits_{\R^{2m}}v(\msnop, \mq) V_{n'} (\msnop,
\mq) d\msnop d\mq\,.
\end{equation}

On the other hand, the boundedness of $(F_k)$ in $\Ld{\R^{2m}}$
enables us to define a bounded sequence of operators $\mu_k\in
\Ldws{\R^{2m}; {\cal M}_{b}(\Rd\times \Pd)}$:
$$
\mu_k (\msnop, \mq)(\phi): =\Dup{\mu_k^{\msnop\mq}}{\phi}, \quad \phi \in \Cnl{\Rd\times \Pd}.
$$

Therefore, there exists  a subsequence $(\mu_{k'})\subseteq(\mu_k)$
such that $\mu_{k'}\def\Dscon{\relbar\joinrel\dscon} \povrhsk\ast \mu$ in  $\Ldws{\R^{2m}; {\cal
M}_{b}(\Rd\times \Pd)}$. By passing to the limit on the left side
of \eqref{jan2618}, we get the relation \eqref{jan0218}.

%Positivity of the  H-measure $\mu$ in the sense of
%\eqref{trcl} follows immediately from \eqref{jan0218} by taking
%$v(p,q)=\rho(p)\rho(q)$, $\rho\in L^2(\R^m)$.
\end{proof}

\begin{remark}
\label{remark_3}

Notice that the last theorem remains valid in the case when the test
functions $\varphi_{1,2}$ depend on the velocity variable ($\msnop$
or $\mq$) as well, i.e. when $\varphi_{1,2}$ are taken from the
space $\Ldc{\R^{m};\Cnl\Rd}$ (with function  $v$ removed from \eqref{jan0218}). As it is enough to prove the statement for test functions from a dense set, we take arbitrary $\varphi_{1,2}\in \Ldc{\R^{m};\Cnl\Rd}$ compactly supported in $\mx$ and approximate them  by sums $\sum\limits_{l=1}^{N} v^l_1(\msnop)\varphi_1^l(\mx)$ and $\sum\limits_{j=1}^{N}v_2^j(\mq) \varphi_2^j(\mx)$ such that
\begin{align*}
&\|\sum\limits_{l=1}^{N} v^l_1\otimes\varphi_1^l-\varphi_1\|_{\Ld{\R^{m};\Cnl\Rd}}\leq 1/N,\\
&\|\sum\limits_{j=1}^{N}v_2^j\otimes \varphi_2^j -\varphi_2\|_{\Ld{\R^{m};\Cnl\Rd}}\leq 1/N.
\end{align*}
\vfill\eject

Then it holds for any $\psi\in \Cp\Pd$
\begin{align*}
&\Bigg|\int\limits_{\R^{2m}}\langle \mu(\msnop,\mq,\cdot,\cdot),\varphi_1(\cdot,\msnop)\bar{\varphi}_2(\cdot,\mq) \otimes\psi\rangle d\msnop d\mq\\&\qquad
-\lim\limits_{n'}\int\limits_{\R^{2m}}\int\limits_{\R^d}
\Bigl({\cal A}_{\psi_\Pd} \,(\ph_1 u_{n'})(\cdot, \msnop)\Bigr)(\mx)
\,\overline{(\ph_2 u_{n'})}(\mx,\mq ) d\mx d\msnop d\mq\;\Bigg|\\
&\leq \Bigg|\int\limits_{\R^{2m}}\Dupp{ \mu(\msnop,\mq,\cdot,\cdot)}{\varphi_1(\cdot,\msnop)\bar{\varphi}_2(\cdot,\mq) \otimes\psi
-\Bigl(\sum\limits_{l=1}^{N} v^l_1(\msnop)\otimes\varphi_1^l\Bigr) \Bigl(\sum\limits_{j=1}^{N} \overline{v_2^j(\mq)\otimes\varphi_2^j}\Bigr) \otimes\psi } d\msnop d\mq \Big|\\
&+\Bigg|\lim\limits_{n'}\int\limits_{\R^{2m}}\int\limits_{\R^d}
\Biggl(\sum\limits_{l,j}^{N} v^l_1(\msnop) \bar v_2^j(\mq)
\Bigl({\cal A}_{\psi_\Pd} \,\ph_1^l u_{n'}(\cdot, \msnop)\Bigr)(\mx)
\,\overline{(\ph_2^j u_{n'})}(\mx,\mq )
\\&\qquad\qquad\qquad- \Bigl({\cal A}_{\psi_\Pd} \,(\ph_1 u_{n'})(\cdot, \msnop)\Bigr)(\mx)
\,\overline{(\ph_2 u_{n'})}(\mx,\mq )\Biggr) d\mx d\msnop d\mq\Bigg|%{\cal O}(1/N)
\\&=\lim\limits_{n'}\Bigg|\int\limits_{\R^{2m}}\int\limits_{\R^d}
\sum\limits_{l,j}^{N}
\Bigl({\cal A}_{\psi_\Pd} \left(( v^l_1\varphi_1^l-\varphi_1) u_{n'}\right)(\cdot, \msnop)\Bigr)\!(\mx)
 \overline{\left(( v_2^j\varphi_2^j-\varphi_2)u_{n'}\right)}(\mx,\mq )d\mx d\msnop d\mq\Bigg|\\
&+
\lim\limits_{n'}\Bigg|\int\limits_{\R^{2m}}\int\limits_{\R^d}
\sum\limits_{l}^{N}
\Bigl({\cal A}_{\psi_\Pd} \left(( v^l_1\varphi_1^l-\varphi_1) u_{n'}\right)(\cdot, \msnop)\Bigr)(\mx)
\overline{(\ph_2 u_{n'})}(\mx,\mq )
d\mx d\msnop d\mq \Bigg|\\
&+
\lim\limits_{n'}\Bigg|\int\limits_{\R^{2m}}\int\limits_{\R^d}
\sum\limits_{j}^{N}
\Bigl({\cal A}_{\psi_\Pd} \,\ph_1 u_{n'}(\cdot, \msnop)\Bigr)(\mx)
 \overline{\left(( v_2^j\varphi_2^j-\varphi_2)u_{n'}\right)}(\mx,\mq )
d\mx d\msnop d\mq\Bigg|
+{\cal O}(1/N)\\&={\cal O}(1/N),
\end{align*}
which proves the remark.

\end{remark}

Now, we shall describe the object $\mu$  in Theorem
\ref{lfeb518} more precisely by showing that it can be represented
as
$\mu(\msnop, \mq,\cdot)=f(\msnop, \mq,\cdot)\nu$, where $\nu\in {\cal
M}_{b}(\R^d\times \Pd)$ is a positive Radon measure, and $f\in
\Ld{\R^{2m}; \Lj{\Rd\times \Pd:\nu}}$. If we could conclude
that for every $\phi\in \Cnl{\R^d\times \Pd}$, the function
$\Dup{\mu(\msnop, \mq, \cdot)}\phi$ represents a kernel of a trace class operator,
then we could rely on \cite[Proposition A.1.]{Ger} to state the
latter representation. The most famous sufficient condition for a
function to be a kernel of a trace class operator is given by the
Mercer theorem. It demands the kernel to be continuous, symmetric
and positive definite. The function $\Dup{\mu(\msnop, \mq, \cdot)}\phi)$ has the last
two properties, but it is not necessarily continuous. Therefore, we
need the following proposition.

\begin{proposition}
\label{prop_repr} The operator $\mu\in \Ldws{\R^{2m}; {\cal
M}_{b}(\Rd\times \Pd)}$  in Theorem \ref{lfeb518} has the
form
\begin{equation}
\label{repr_1} \mu(\msnop, \mq,\mx,\mxi)=f(\msnop, \mq,\mx,\mxi)\nu(\mx,\mxi),
\end{equation}
where $\nu\in {\cal M}_b(\R^d\times \Pd)$ is a non-negative scalar Radon measure, while
$f$ is a function from $\Ld{\R^{2m}; \Lj{\Rd\times \Pd:\nu}}$ satisfying
$$\int_{\R^{2m}}\int_{\R^d\times \Pd}\rho(\msnop)\bar \rho(\mq)\phi(\mx,\mxi)f(\msnop, \mq,\mx,\mxi)d\nu(\mx,\mxi) d\msnop d\mq\geq 0$$
for any $\rho\in \Ldc{\R^m}$, $\phi\in \Cnl{\R^d\times \Pd}$,
$\phi\geq 0$.
\end{proposition}

\begin{proof}
The proof is based on rewriting the measure $\mu(\msnop,
\mq,\mx,\mxi)$ via the basis in the (Hilbert) space $\Ld{\R^{2m}}$.

Accordingly, let $\{e_i\}_{i\in \N}$ be an orthonormal basis in
$\Ld{\R^m}$. Denote by $\mu_{ij}\in {\cal M}_b(\R^d\times \Pd)$ an
${\rm H}_\Pd$-measure generated by the sequences $\int_{\R^m}e_i(\msnop)u_n(\mx,\msnop)d\msnop$
and $\int_{\R^m}e_j(\mq)u_n(\mx,\mq)d\mq$. We claim:
\begin{equation}
\label{hs-repr} \mu(\msnop, \mq,\mx,\mxi)=\sum\limits_{i,j=1}^\infty
\mu_{ij}(\mx,\mxi)\bar e_i(\msnop) e_j(\mq).
\end{equation}
Indeed, take arbitrary $\rho\in \Ld{\R^{2m}}$, $\varphi_{1,2}\in \Cnl{\R^d}$, $\psi\in \Cp\Pd$, and notice that
\begin{equation*}
\rho(\msnop, \mq)=\sum\limits_{i,j=1}^\infty c_{ij}
e_i(\msnop)\bar e_j(\mq),
\end{equation*} where $(c_{ij})_{i,j\in \N}$ is a square sumable sequence.

According to the definition of the functional $\mu$, we have
\begin{align*}
&\int_{\R^{2m}}\rho(\msnop, \mq)\langle \mu(\msnop, \mq,\cdot),\varphi_1\bar\varphi_2\psi \rangle d\msnop d\mq\\
&=\lim\limits_{n\to \infty}\int_{\R^{2m}}\int_{\R^d} \rho(\msnop, \mq)
\Bigl({\cal A}_{\psi}\,\varphi_1 u_n(\cdot,\msnop)\Bigr)(\mx)
\,\bar\varphi_2 \bar u_n(\mx,\mq) d\mx d\msnop d\mq\\
&=\sum\limits_{i,j=1}^\infty\lim\limits_{n\to \infty}\int_{\R^d} \!c_{ij}
\Biggl({\cal A}_{\psi}\,\varphi_1 \int_{\R^m} u_n(\cdot,\msnop)e_i(\msnop)d\msnop  \Biggr)(\mx)\,
\bar \varphi_2 (\mx)\int_{\R^m} \!\bar u_n(\mx,\mq)\bar e_j(\mq)d\mq d\mx\\
&=\sum\limits_{i,j=1}^\infty c_{ij} \langle
\mu_{ij}(\mx,\mxi),\varphi_1\bar\varphi_2\psi\rangle
=\sum\limits_{i,j=1}^\infty  \langle
\mu_{ij}(\mx,\mxi),\varphi_1\bar\varphi_2\psi\rangle
\int_{\R^{2m}}\rho(\msnop, \mq)\bar e_i(\msnop) e_j(\mq)d\msnop d\mq,
\end{align*}
which completes the proof of \eqref{hs-repr}. Remark that in the
last derivation we have used the square integrability of the
sequence $(c_{ij})$ and the Lebesgue dominated convergence theorem.

We introduce a positive bounded measure
$$
\nu(\mx,\mxi)=\sum\limits_{i=1}^\infty \frac{1}{2^i} \mu_{ii}(\mx,\mxi),
$$
as the weighted trace of the measure matrix $(\mu_{ij})_{i,j=1,\infty}$ and claim that
$$
\mu(\msnop, \mq,\cdot) << \nu,
$$for almost every $\msnop, \mq\in \R^m$. Indeed, if $\nu(E)=0$ for some Borel set $E\subset \R^d\times\Pd$, then
$\mu_{ii}(E)=0$ for every $i\in \N$. On the other hand, due to the hermitian character of matrix ${\rm H}_\Pd$ measures, we have
$$
|\mu_{ij}(E)|\leq \mu_{ii}(E)^{1/2}\mu_{jj}(E)^{1/2}.
$$
From here, it follows that $\mu_{ij}(E)=0$ for every $i,j\in \N$, and thus, according to \eqref{hs-repr}, $\mu(\msnop, \mq,\mx,\mxi)(E)=0$ for almost every $\msnop, \mq\in \R^m$.

Now, the conclusion follows from the Radon-Nikodym theorem.

\end{proof}

%\begin{remark}
%\label{rem_imp2}

Next, we shall make an extension of Theorem \ref{lfeb518}.

%First, notice that if in Theorem \ref{lfeb518} we assume
%$\varphi_{1,2}\in C_c(\R^d)$  (compactly supported functions), we
%can assume that the sequence $(u_n)$ from \eqref{main-sys} is only
%locally in $L^s(\R^d)$, $s\geq 2$, i.e. $(u_n)\in L^s_{loc}(\R^d)$ .

Notice that if in Theorem \ref {H-meas} we assume $u_n \in \pL
s{\Rd}$ for some $s>2$, then the $\R^d$--projection of a corresponding
H-measure   is absolutely continuous with
respect to the Lebesgue measure (see \cite[Corollary 1.5]{Tar} and
\cite[Remark 2, a)]{pan_arma}). Furthermore, in that case we can
assume that the test function $\varphi_{1}$ is merely in $\pL
{r}{\R^d}$.

The result generalises to sequences of functions taking values in a function space. More precisely, the following theorem holds.

\begin{theorem}
\label{lfeb518-r}
Assume that the sequence $(u_n)=(u_n(\mx,\msnop))$,
 converges weakly to zero in $\Ld{\R^{m+d}}\cap \Ld{\Rm; \pL s{\Rd}}$, $s>2$. Then the $\R^d$ projection $\int_{\Pd}d\nu(\mx,\mxi)$ of the measure $\nu$ from the last proposition can be extended to a bounded functional on  $\pL{r}\Rd$, where $r$ is the dual index of $s/2$.
Furthermore, for all $\varphi_1\in \Ld{\Rm; \pL {r}{\Rd}}
$,
$\varphi_2\in \Ldc{\R^{m};\Cnl\Rd}$,  and $\psi\in \pC d{\Pd}$, it holds:
\begin{equation}
\label{jan0218-r}
\begin{split}
\lim\limits_{n'}\int\limits_{\R^{2m}}\int\limits_{\Rd}
&(\ph_1 u_{n'})(\mx,
\msnop)
\,\Bigl(\overline{{\cal A}_{\psi_\Pd} \,\ph_2 u_{n'}(\cdot,
\mq)}\Bigr)(\mx) d\mx d\msnop d\mq\\
&=\int\limits_{\R^{2m}} \langle
\mu(\msnop,\mq,\cdot,\cdot),{\varphi}_1(\cdot,\msnop)\bar{\varphi}_2
(\cdot,\mq)\otimes\bar\psi\rangle d\msnop d\mq.
\end{split}
\end{equation}
\end{theorem}

\begin{proof}

%First, we shall show that we can extend the object $\mu$  in
%Theorem \ref{lfeb518}  as a continuous bilinear form  on
%$\Ld{\R^{m}; \pL r{\Rd}} \times\Ldc{\R^{m};\Cnl\Rd}\times \pC d{\Pd}$.
% To this end,
%notice first that ${\rm supp}f(\cdot,\cdot,\mx,\mxi)\subset
%\Omega\subset\subset \R^{2m}$ independently on $(\mx,\mxi)\in
%\R^d\times P$ since the sequence $(u_n)$ defining the H-measure
%$\mu$ is uniformly compactly supported with respect to $p\in \R^m$.

Let $\varphi_1^\eps\in \Cc{\R^{m+d}}$, $\eps>0$, be a family of
continuous functions such that $\|\varphi_1 -\varphi_1^\eps
\|_{\Ld{\Rm; \pL {r}{\Rd}}}\to 0$ as $\eps\to 0$. By means of Remarks \ref{rem1P} and \ref{remark_3} we define

\begin{align}
\label{cshy}
\begin{split}
&\int\limits_{\R^{2m}} \langle
\mu(\msnop,\mq,\cdot,\cdot),\varphi_1\bar{\varphi}_2
\otimes\bar\psi\rangle d\msnop d\mq: =\lim\limits_{\eps\to
0}\int\limits_{\R^{2m}} \langle
\mu(\msnop,\mq,\cdot,\cdot),\varphi^\eps_1\bar{\varphi}_2
\otimes\bar\psi\rangle d\msnop d\mq\\
&=\lim\limits_{\eps\to 0}\lim\limits_{n'\to \infty}\int\limits_{\R^{2m}}\int\limits_{\Rd}
(\ph_1^\eps u_{n'})(\mx,
\msnop)
\,\Bigl(\overline{{\cal A}_{\psi_\Pd} \,\ph_2 u_{n'}(\cdot,
\mq)}\Bigr)(\mx) d\mx d\msnop d\mq.
\end{split}
\end{align}
The latter limit exists since for $\eps_1,\eps_2>0$ it
holds

\begin{align}
\nonumber
&\Bigg|\,\int\limits_{\R^{2m}} \langle
\mu(\msnop,\mq,\cdot,\cdot),(\varphi_1^{\eps_1}-\varphi_1^{\eps_2})(\cdot,\msnop)\bar{\varphi}_2(\cdot,\mq)
\otimes\bar\psi\rangle d\msnop d\mq\Bigg|\\
&\leq\limsup\limits_{n'} \int\limits_{\R^{2m}}\int\limits_{\Rd}
\Big|  (\varphi_1^{\eps_1}-\varphi_1^{\eps_2} ) u_{n'}(\mx,\msnop)
\,\Bigl(\overline{{\cal A}_{\psi_\Pd} \,\ph_2 u_{n'}(\cdot,
\mq)}\Bigr)(\mx) \Big| d\mx d\msnop d\mq \nonumber\\& \leq
\limsup\limits_{n'} C \int_{\R^{2m}}  \|
(\varphi_1^{\eps_1}-\varphi_1^{\eps_2} ) u_{n'}(\cdot,\msnop)
\|_{\pL {s'}{\Rd}}   \|(\varphi_2 u_{n'})(\cdot,\mq)\|_{\pL s{\Rd}}
d\msnop d\mq
\nonumber\\
& \leq \limsup\limits_{n'} C \int_{\R^{m}}
\|(\varphi_1^{\eps_1}-\varphi_1^{\eps_2})(\cdot,\msnop)\|_{\pL
r{\Rd}} \|( u_{n'}(\cdot,\msnop)\|_{\pL s{\Rd}} d\msnop
\nonumber\\&\hskip26mm \cdot
 \int_{\R^{m}}\|\varphi_2 (\cdot,\mq)\|_{\Lb{\Rd}}\| u_{n'}(\cdot,\mq)\|_{\pL s{\Rd}}  d\mq
\nonumber\\
& \leq \limsup\limits_{n'} C \,\,
\|(\varphi_1^{\eps_1}-\varphi_1^{\eps_2})\|_{\Ld{\Rm; \pL r{\Rd}}}
\|\varphi_2 \|_{\Ld{\Rm; \Lb{\Rd}}} \,\| u_{n'}\|_{\Ld{\Rm; \pL
s{\Rd}}}^2 \,, \nonumber
\end{align}
where $C$ depends on $s$, $d$, and $\|\psi\|_{\pC d{\Pd}}$. Since
$\|\varphi_1-\varphi_1^\eps\|_{\Ld{\Rm; \pL {r}{\Rd}}}\to 0$, the
limit in \eqref{cshy} exists.

The same analysis from the above implies
\begin{equation*}
\lim\limits_{\eps\to 0 }\int\limits_{\R^{2m}}\int\limits_{\Rd}
\Big|  (\varphi_1^{\eps}-\varphi_1 ) u_{n'}(\mx,\msnop)
\,\Bigl(\overline{{\cal A}_{\psi_\Pd} \,\ph_2 u_{n'}(\cdot,
\mq)}\Bigr)(\mx) \Big|
d\mx d\msnop d\mq = 0
\end{equation*}
and the convergence is uniform with respect to $n'$. Thus we can exchange limits in the second line of \eqref{cshy}, which proves \eqref{jan0218-r}.

In order to prove that  the $\R^d$ projection $\int_{\Pd}d\nu(\mx,\mxi)$ of
the measure $\nu$ belongs to $\pL{r'}\Rd$ take an
arbitrary $\varphi\in  \Cc\Rd$ and consider
\begin{align*}
&\Big|\int_{\R^d} \varphi(\mx)\int_{\Pd}d\nu(\mx,\mxi) \bigg|
=\sum\limits_{i=1}^\infty \Big| \langle \frac{1}{2^i}\mu^{ii},\varphi \otimes 1\rangle \Big|\\
&\leq\sum\limits_{i=1}^\infty\frac{1}{2^i}\lim\limits_{n'\to
\infty}\int_{\R^{2m}}\int_{\Rd}\Big|\varphi(\mx)u_{n'}(\mx,\msnop)
e_i(\msnop)\bar u_{n'}(\mx,\mq)\bar e_i(\mq) \Big|d\mx d\msnop d\mq\\
&\leq
\sum\limits_{i=1}^\infty\frac{1}{2^i}\|\varphi\|_{\pL{r}\Rd}
\limsup\limits_{n'\to \infty}\|u_{n'}\|^2_{\Ld{\R^{2m}; \pL s{\R^d}}}\leq C\|\varphi\|_{\pL{r}\Rd}.
\end{align*}
Thus $\int_{\Pd}d\nu(\mx,\mxi)$ can be extended to a bounded functional on  $\pL{r}\Rd$, i.e. there exists an $h\in \pL{r'}\Rd$ such that $\int_{\Pd}d\nu(\mx,\mxi) =h(\mx) d\mx$.

\end{proof}

The following statement on the measure $\nu$ now follows from   results on slicing measures \cite[Theorem 1.5.1]{Evans}.

\begin{lemma}
\label{lemma-slicing}
Under assumptions of the last theorem, for $\hbox{\rm a.e. } \mx\in \Rd$ there exists a Radon probability measure $\nu_\mx$ such that $d\nu(\mx, \mxi)= d \nu_\mx  (\mxi) h(\mx) d\mx$, where $h$ is a $\LLp {r'}$ function introduced above. More precisely, for each $\phi\in\Cnl{\Rd\times \Pd}$
$$
\int_{\Rd\times \Pd} \phi(\mx, \mxi)d\nu(\mx, \mxi) = \int_\Rd \left( \int_\Pd \phi(\mx, \mxi) d \nu_\mx  (\mxi) \right)h(\mx) d\mx\,.
$$
The above result is also valid if we take a test function $\phi \in \pL{r}{\Rd; {\Cp \Pd}}$.
\end{lemma}
%\end{remark}

\section{Proof of the main theorem}

In this section, we shall prove Theorem \ref{main-result}. The proof
is based on the special choice of the test function to be applied in
\eqref{defws}, and the ${\rm H}_\Pd$-measures techniques developed in the
previous section.

We introduce the multiplier operator ${\cal I}$ with the symbol
$\frac{1-\theta(\mxi)}{\left(|\xi_1|^{l\alpha_1}+\dots+|\xi_d|^{l\alpha_d}\right)^{1/l}}$, where
$\theta\in\Cbc\Rd$ is a cut-off function, such that $\theta\equiv1$ on a neighbourhood of the origin.

According to Lemma \ref{marz}, for any $\psi\in \pC d\Pd$, the
multiplier operator ${\cal I} \circ {\cal A}_{\psi\circ\pi_{\Pd}}:
\Ld\Rd\cap \pL{s}\Rd \to \W\malpha{s}{\R^d}$
 is bounded (with $\LLp {s}$ norm considered on the domain). Indeed, it is enough to notice
that the symbol of
$\pa^{\alpha}_{x_k}\left({\cal I}\circ {\cal A}_{\psi\circ\pi_{\Pd}}\right)$:
$$
(\psi\circ\pi)(\mxi)\frac{(1-\theta(\mxi))(2\pi i\xi_k)^{\alpha_k}}{\left(|\xi_1|^{l\alpha_1}+\dots+|\xi_d|^{l\alpha_d}\right)^{1/l}}
$$
is a smooth, bounded function that satisfies conditions of Lemma \ref{m1}.

Insert in \eqref{defws} (with reintroduced sub-index $n$) the test function $g_n$ given by (a
similar procedure was firstly applied in \cite{sazh}):
\begin{equation}
\label{jul0218} g_n(\mx,\msnop)=\rho_1(\msnop)\int_{\Rm}
({\cal I}\circ {\cal A}_{\psi\circ\pi_{\Pd}})\bigl(\ph u_n(\cdot, \mq)\bigr)(\mx)\rho_2(\mq)d\mq,
\end{equation}
where  $\psi\in \pC d\Pd$, $\ph\in \Cbc\Rd$, $\rho_1, \rho_2 \in {\rm
C}_c^{|\mkappa|}(\R^m)$, and $\mkappa$ is the multi-index appearing
in \eqref{main-sys}. Due to the boundedness properties of the
operator ${\cal I}\circ {\cal A}_{\psi\circ\pi_{\Pd}}$ discussed
above, the sequence $(g_n)$ is bounded in ${\rm
C}_c^{|\mkappa|}(\R^m) \times \W\malpha s\Rd$.

Letting $n\to \infty$ in \eqref{defws}, we
get after taking into account Theorem \ref{lfeb518-r} and the
strong convergence of $(G_n)$
\begin{equation*}
%\label{jul0318}
\begin{split}
\int_{\R^{2m}}\int_{\R^d\times \Pd}
A(\mx,\mxi,\msnop)\overline{\rho_1(\msnop)\rho_2(\mq){\ph}(x)\psi(\mxi)}d\mu(\msnop,\mq,\mx,\mxi)d\msnop d\mq=0,
\end{split}
\end{equation*}
where, let it be repeated,  $A(\mx,\mxi,\msnop)=\sum_{k=1}^d (2\pi i\xi_k)^{\alpha_k}
a_k$. As  the test functions $\rho_i$, $\ph$, and $\psi$ are taken from dense subsets in appropriate spaces,  we conclude
\begin{equation}
\label{jul0328} A(\mx,\mxi,\msnop)
d\mu(\msnop,\mq,\mx,\mxi)=0, \quad \ae{\msnop, \mq \in\R^{2m}}\,.
\end{equation}

For $s=2$ the non-degeneracy condition \eqref{kingnl} directly implies that $\mu=0$.
In order to show the same result for $s>2$ fix an arbitrary $\delta>0$, and for a  $\rho\in
\Ldc{\R^m}$ and $\phi \in \Cnl{\R^d\times \Pd}$ consider the test
function
\begin{equation*}
%\label{tfnova}
\frac{\rho(\msnop)\bar\rho(\mq)\phi(\mx,\mxi)\overline{A(\mx,\mxi,\msnop)}}{|A(\mx,\mxi,\msnop)|^2
+\delta}.
\end{equation*}
From \eqref{jul0328}, we obtain

\begin{equation*} \int_{\R^{2m}}\int_{\R^d\times
\Pd}\frac{\rho(\msnop)\bar\rho(\mq)\phi(\mx,\mxi)|A(\mx,\mxi,\msnop)|^2}{
|A(\mx,\mxi,\msnop)|^2+\delta} d\mu(\msnop,\mq,\mx,\mxi) d\msnop d\mq=0,
\end{equation*}
which by means of representation \eqref{repr_1} and Fubini's theorem takes the form
\begin{equation}
\label{fin_4} \int_{\R^d\times \Pd}\int_{\R^{2m}}
\frac{\rho(\msnop)\bar\rho(\mq)\phi(\mx,\mxi)|A(\mx,\mxi,\msnop)|^2}{ |A(\mx,\mxi,\msnop)|^2
+\delta}f(\msnop,\mq,\mx,\mxi) d\msnop d\mq d\nu(\mx,\mxi)=0.
\end{equation}
Let us denote
$$
I_\delta(\mx, \mxi)=  \int_{\R^{2m}}
\rho(\msnop)\bar\rho(\mq)\frac{|A(\mx,\mxi,\msnop)|^2}{ |A(\mx,\mxi,\msnop)|^2
+\delta}f(\msnop,\mq,\mx,\mxi)d\msnop d\mq\,.
$$
According to the non-degeneracy condition \eqref{kingnl} and the representation
of the measure $\nu$ given in Lemma \ref{lemma-slicing}, for $s>2$  we have
$$
I_\delta(\mx, \mxi)\to   \int_{\R^{2m}}
\rho(\msnop)\bar\rho(\mq)f(\msnop,\mq,\mx,\mxi)d\msnop d\mq,
$$
as $\delta\to 0$ for $\nu- {\rm a.e. }\, (\mx, \mxi) \in \R^d\times \Pd$.
By using the Lebesgue dominated convergence theorem, it follows from
\eqref{fin_4} after letting $\delta\to 0$:
\begin{align*}
%\label{fin_5}
&\int_{\R^d\times
\Pd}\int_{\R^{2m}}\rho(\msnop)\bar\rho(\mq)\phi(\mx,\mxi)f(\msnop,\mq,\mx,\mxi) d\msnop d\mq d\nu(\mx,\mxi)
\\
&=\int_{\R^{2m}}\rho(\msnop)\bar\rho(\mq)\,\langle\mu(\msnop,\mq,\cdot,\cdot), \phi\rangle d\msnop d\mq=0. \nonumber
\end{align*}
Having in mind the definition of the measure $\mu$
from Theorem \ref{lfeb518}, by putting here
$\phi(\mx,\mxi)=|\varphi(\mx)|^2$ for $\varphi\in \Cnl{\R^d}$, we immediately obtain
$$
\lim\limits_{n'\to
\infty}\int_{\R^d}\left|\int_{\Rm}\rho(\msnop)u_{n'}(\mx,\msnop)d\msnop \right|^2
|\varphi(\mx)|^2d\mx =0.
$$
Due to arbitrariness of $\varphi$, this concludes the proof. $\Box$

\begin{remark}
We conclude the section by remarking that our results easily extend
to equations containing mixed derivatives with respect to the space variables  (see also \cite[Theorem
2.1]{Ger}):

\begin{equation}
\label{general}
\begin{split}
{\cal P}u_n(\mx,\msnop)&=\sum\limits_{s\in I}\partial^{\malpha_s}_\mx
\left(a_s(\mx,\msnop) u_n(\mx,\msnop)\right)=\pa^\mkappa_\msnop
G_n(\mx,\msnop),
\end{split}
\end{equation} where $I$ is a finite set of indices, and $\partial_\mx^{\malpha_s}=\partial_{x_1}^{\alpha_{1s}}\dots \partial_{x_d}^{\alpha_{ds}}$,
for a multi-index $\malpha_s=(\alpha_{s1},\dots,\alpha_{sd})\in
\R^d$.

Denote by $A$ the principal symbol of the
(pseudo-)differen\-tial operator ${\cal P}$, which is of the form
$$
A(\mxi,\mx,\msnop)=\sum\limits_{s\in I'} (2\pi i\mxi)^{\malpha_s}a_s(\mx,\msnop),
$$
where the upper sum goes above  all terms from \eqref{general} whose order of derivative $\malpha_s$ is not dominated by any other multiindex from $I$.

For $A$ we must additionally assume  that there exist $\alpha_1, \cdots, \alpha_d\in \Rpl$ such that for any positive
$\lambda\in \R$, it holds
$$
A(\lambda^{1/\alpha_1}\xi_1,\dots,\lambda^{1/\alpha_d}\xi_d,\mx,\msnop)=
\lambda A(\mxi,\mx,\msnop),
$$
and that it satisfies
genuine non-degeneracy condition: for almost every $x\in \R^d$, every
$\xi\in \Pd$, it holds
\begin{equation*}
A(\mx,\mxi,\msnop)\not= 0 \quad\ae{ \msnop \in \Rm}\,;
\end{equation*}

 The proof of Theorem \ref{main-result} for equation of form
\eqref{general}  goes along the same lines as for the equation
\eqref{main-sys}.

\end{remark}

\begin{remark}
Let us finally remark that in the case when derivative orders $\alpha_k$,
$k=1,\dots,d$, are non-negative integers, we can assume that the
sequence $u_n$ is only locally bounded in
$\Ldl{\R^m; \Ll s\Rd}$.

In that case we simply take

\begin{equation*}
g_n(\mx,\msnop)=\rho_1(\msnop)\ph(\mx) \int_{\Rm} ({\cal I}\circ
{\cal A}_{\psi\circ\pi_{\Pd}})\bigl(\ph u_n(\cdot,
\mq)\bigr)(\mx)\rho_2(\mq)d\mq,
\end{equation*} instead of $g_n$ from \eqref{jul0218}. By repeating the
rest of the procedure from this section, we conclude that the
measure $\mu$ from Theorem \ref{lfeb518} corresponding to $(\ph u_n)$ equals zero. Due to
arbitrariness of $\ph$, we conclude that for any $\rho\in \Ldc{\R^m}$, the sequence $(\int u_n(\mx,\msnop) \rho(\msnop)d\msnop)$ is strongly
precompact in $\Ldl\Rd$.

\end{remark}

\section{Ultra-parabolic equation with discontinuous coefficients}

In this section, we consider an ultra-parabolic equation with
discontinuous coefficients in a domain $\Omega$ (an open subset of
$\Rd$). Ultraparabolic equations (with regular coefficients) were
first considered by Graetz \cite{Gr} and Nusselt \cite{Nus} in their
investigations concerning the heat transfer. A specific situation modelled by such equations is the one when
diffusion can be neglected in the directions $x_{l+1}, \dots, x_d$,
$l\geq 0$. Recently, such
equations were investigated in \cite{pan_jms} and we aim to extend
results from there.

More precisely, the equation that we are going to
consider here has the form
\begin{equation}
\label{scldf} \dv \vf(\mx,u)- \dv \dv  \mB(\mx,u) +\psi(\mx,u)=0,
\end{equation}
where $\mB(x,u)=(b_{jk})_{j,k=1,\dots,d}$ is a symmetric matrix such that
for some $l<d$ it holds $(b_{jk})\equiv 0$ for $\min(j, k)\leq l$, while
$\tilde{\mB}=(b_{jk})_{j,k=l+1,\dots,d}$   satisfies an
ellipticity condition on $\R^{d-l}$ in the following sense: for
every $\tilde{\mxi}\in \R^{d-l}$, $\lambda_1,\lambda_2\in \R$
and $\mx\in \Omega$,
$$
 \left(\lambda_1>\lambda_2\right) \povlaci
(\tilde{\mB}(\mx,\lambda_1)-\tilde{\mB}(\mx,\lambda_2))\tilde{\mxi} \cdot  \tilde{\mxi}\geq c |\tilde{\mxi}|^2, \ \ c>0.
$$

Accordingly, we shall use anisotropic spaces like $\W{({\sf 1,
2})}q\Omega$, where $({\sf 1, 2})\in \Rd$ is a multiindex with first
$l$ components equal to 1.

Furthermore, we assume that $\psi\in \Lj{\Omega;\Lb\R}$, while $\vf=(f_1,\dots,f_d)$ and $\mB$ are such
that for every $j, k=1,\dots,d$
\begin{equation*}
 \partial_\lambda f_k, \partial_\lambda b_{jk} \in \Ldl{\R;
{\Ll r\Omega}}, \ \ r>1.
\end{equation*}

We also need to assume a kind of uniform continuity of $\vf$ and $\mB$ in the sense that there exists an increasing function $w$ on $\R^+$, vanishing and continuous at $0$ (i.e. $w$ is a modulus of continuity type function), and $\sigma\in \Ll {1+\eps}\Omega, \,\eps>0$ such that
\begin{equation}
\label{uc} |\vf(x,\lambda_1)-\vf(x,\lambda_2)|,
|\mB(x,\lambda_1)-\mB(x,\lambda_2)| \leq
w\left(\big|\lambda_1-\lambda_2)\big|\right)\big|\sigma(x)\big|.
\end{equation}

Concerning regularity with respect to $\mx\in \R^d$ of the functions $\vf$ and $\mB$, we assume that for every $\lambda\in \R$
\begin{equation*}
%\label{p3}
 \dv \vf(\mx,\lambda)-\dv \dv  \mB(\mx,\lambda)=\gamma(\mx,\lambda)\in
{\cal M}(\R^d).
\end{equation*}

To proceed, denote $\gamma(\mx,\lambda)=\omega(\mx,\lambda)d\mx+\gamma^s(\mx,\lambda)$
where $\omega(\mx,\lambda)d\mx$ denotes the regular, and $\gamma^s(\mx,\lambda)$
denotes the singular part of the measure $\gamma$ with respect to the Lebesgue measure. The following
definition is used in \cite{pan_jms}.

\begin{definition}
\label{eac}
We say that a function $u\in \Lb\Omega$ represents an entropy admissible weak solution to \eqref{scldf} if for every $\lambda \in \R$ it holds
\begin{align}
\label{p5}
&\dv \Bigl({\rm sgn}(u(\mx) - \lambda)\bigl(\vf(\mx, u(\mx)) - \vf(\mx, \lambda)\bigr)\Bigr) \\&-
\dv \dv  \Bigl({\rm sgn}(u(\mx) - \lambda)\bigl(\mB(\mx, u(\mx)) - \mB(\mx, \lambda)\bigr)\Bigr) \nonumber\\
&+
{\rm sgn}(u(\mx)- \lambda)\bigl(\omega(\mx,\lambda) + \psi(\mx, u(\mx))\bigr) - |\gamma(\mx,\lambda)| \leq 0
\nonumber
\end{align} in the sense of distributions on $\R^d$.
\end{definition}

We shall prove a result similar to those from \cite{pan_jms, pan_arma}, stating the assumptions under which a sequence of entropy solutions  is strongly
precompact in $\Ldl\Omega$. There it is assumed that $\max\limits_{\lambda\in
\oi {-M}M}|\vf (\cdot,\lambda)|, \max\limits_{\lambda\in
\oi {-M}M}|\mB (\cdot,\lambda)|\in \Ldl\Omega$, for $M=\limsup_n \nor{u_n}{\Lb\Omega}$,  while we demand $ \partial_\lambda \vf, \partial_\lambda \mB \in \Ldl{\R;
{\Ll r\Omega}}$ for an $r>1$. Remark that we have increased regularity with respect to $\lambda\in \R$ (there the continuity is merely assumed), but we have decreased it with respect to $\mx\in \Omega$. However, in the case $r\geq 2$ the statement of the next theorem also follows from the more general results of \cite{pan_jms}.

\begin{theorem}
Assume that the coefficients of equation \eqref{scldf} satisfy the genuine nonlinearity conditions analogical to \eqref{kingnl}:

\begin{itemize}

\item for every $\mxi=(\hat{\mxi},\tilde{\mxi})\in
\Pd=\{(\hat{\mxi},\tilde{\mxi})\in \R^{l}\times \R^{d-l}:\, |\hat{\mxi}|^2+|\tilde{\mxi}|^4=1 \}$ and almost every $x\in \R^d$
\begin{equation}
\label{pgnl}  2\pi i\sum\limits_{k=1}^l \xi_k \pa_\lambda
f_k(\mx,\lambda)+4 \pi^2 \langle \pa_\lambda\mB(\mx,\lambda)\mxi,\mxi
\rangle
\not=0 \quad\ae{\lambda\in \R}\,.
\end{equation}
\end{itemize}

Then, a sequence of entropy solutions $(u_n)$ to \eqref{scldf} such
that $\nor{u_n}{\Lb\Omega}<M$ for every $n\in \N$ is strongly
precompact in $\Ldl\Omega$.

\end{theorem}

\begin{proof}

To prove the theorem, remark first that, according to the Schwartz
theorem \cite[Theorem I.V]{Sw}, for every $\lambda\in \R$ we can rewrite
\eqref{p5} as

\begin{align}
\label{p6}
&\dv \Bigl({\rm sgn}(u_n(\mx) - \lambda)\big(\vf(\mx, u_n(\mx)) - \vf(\mx, \lambda)\big)\Bigr) \\&-
\dv \dv  \Bigl({\rm sgn}(u_n(\mx) - \lambda)\big(\mB(\mx, u_n(\mx)) - \mB(\mx, \lambda)\big)\Bigr) \nonumber\\
&= G_n(\mx,\lambda)
\nonumber
\end{align}  where $G_n(\cdot,\lambda)\in {\cal M}(\Omega
)$ are Radon measure on
$\Omega$, locally uniformly bounded with respect to $n$.
According to \cite[Theorem 1.6]{Evans} (see also
\cite[Proposition 7]{pan_jms}), the sequence of measures
$(G_n(\cdot,\lambda))$ is strongly precompact in
$\Wl{({\sf -1, -2})}q\Omega$ for each $q\in \oi 1{\frac{d}{d-1}}$.
Furthermore, for every $\varphi\in
C_c^{{\sf 1,2}}(\Omega)$, we have according to \eqref{uc}
\begin{align}
\label{1/2}
&|\langle
G_n(\cdot,\lambda_1)-G_n(\cdot,\lambda_2), \varphi \rangle|
\\&
=\int_{\Omega}\Big|\Big({\rm sgn}(u_n\!-\!\lambda_1)\big(\mB(\cdot,u_n)\!-\! \mB(\cdot,\lambda_1)\big)\!-\!{\rm sgn}(u_n-\lambda_2)\big(\mB(\cdot,u_n)\!-\! \mB(\cdot,\lambda_2)\bigr) \Bigr)\cdot(\nabla\otimes\nabla) \varphi \Big| d\mx
 \nonumber\\&
\qquad +\int_{\Omega}\Big|\Big({\rm sgn}(u_n\!-\!\lambda_1)\big(\vf(\cdot,u_n)\!-\! \vf(\cdot,\lambda_1)\big)\!-\!{\rm sgn}(u_n\!-\!\lambda_2)\big(\vf(\cdot,u_n)\!-\!\vf(\cdot,\lambda_2)\big)\Bigr)\cdot\nabla \varphi \Big|d\mx
\nonumber
 \\& \leq
C w\left(\big|\lambda_1-\lambda_2)\big|\right) \nor\varphi{\WW{({\sf 1, 2})}{q'}},
\nonumber
\end{align}
for a constant $C$ independent of $n$ (it
depends only on $f$, $\mB$, and $\sigma$). Indeed, according to \eqref{uc} it holds
\begin{align*}
&\big|{\rm sgn}(u-\lambda_1)(\vf(\mx, u)- \vf(\mx, \lambda_1)-{\rm sgn}(u-\lambda_2)(\vf(\mx, u)- \vf(\mx, \lambda_2) \big| \quad  \\
 &\leq \begin{cases}
\aps{ \vf(\mx, \lambda_1)-\vf(\mx, \lambda_2)},  &(u-\lambda_1)(u-\lambda_2)\geq 0\\
\aps{ \vf(\mx, u)-\vf(\mx, \lambda_1)} +  \Apslr{\vf(\mx, u)-\vf(\mx, \lambda_2) },  &(u-\lambda_1)(u-\lambda_2)\leq 0
\end{cases}    \\
&\leq\;2 w\left(\big|\lambda_1-\lambda_2)\big|\right) \big|\sigma(\mx)\big|,
\end{align*}
and similarly for $\vf$ replaced by $\mB$,
from where \eqref{1/2} immediately follows.

Take now a countable dense subset $D$ of $\R$ and for every
$\lambda_m\in D$ denote by $G(\cdot,\lambda_m)\in {\cal M}(\Omega)$
such that $G_n(\cdot,\lambda_m)\dstr G(\cdot,\lambda_m)$ strongly in
$\Wl{({\sf -1, -2})}q\Omega$  along a
subsequence.  Since $D$ is countable, we can choose the same subsequence (which we denote the same as the original one) for every
$\lambda_m\in D$. Now, we extend
$G(\cdot,\lambda)$, $\lambda\in D$, by continuity on entire $\R$: for
every $\lambda\in \R$, we choose a sequence $(\lambda_m)$ from $D$
converging to $\lambda$ and define for every $\varphi\in C_c^{{\sf 1,2}}(\Omega)$:
\begin{equation}
\label{wd}
\langle G(\cdot,\lambda), \varphi \rangle:=\lim\limits_{m\to \infty} \langle  G(\cdot,\lambda_m), \varphi \rangle.
\end{equation} The latter is well defined  since for any $\lambda_{1},\lambda_{2}\in D$ and any $\eps>0$ one can find an $n>0$ such that
\begin{align*}
&\Big| \langle G(\cdot,\lambda_{1})-G(\cdot,\lambda_{2}), \varphi \rangle \Big| \leq \\&
\Big| \langle G_n(\cdot,\lambda_{1})-G(\cdot,\lambda_{1}), \varphi\rangle \Big|+
\Big| \langle G_n(\cdot, \lambda_{1})-G_n(\cdot,\lambda_{2}), \varphi \rangle \Big|
+\\& | \langle G_n(\cdot,\lambda_{2})-G(\cdot,\lambda_{2}), \varphi \rangle \Big|\leq
\Big(\eps+C w\left(\big|\lambda_1-\lambda_2)\big|\right)+\eps\Big)\nor\varphi{\WW{({\sf 1, 2})}{q'}} .
\end{align*}
From here, the Cauchy criterion will provide properness of \eqref{wd}. Furthermore, since $G(\cdot,\lambda_m)$ are Radon measures, the functional $G(\cdot, \lambda)$ is also a Radon measure.

Using the same arguments, it is not difficult to prove that for every $\lambda \in \R$,
$$
G_n(\cdot,\lambda)\to G(\cdot,\lambda) \ \ {\rm in} \ \
\Wl{({\sf -1, -2})}q\Omega.
$$

According to the Lebesgue dominated convergence theorem, we
conclude from the latter that
\begin{equation}
\label{uvjet51} G_n\to G \ \ {\rm in} \ \
\Ldl{\R;\Wl{({\sf -1, -2})}q\Omega}.
\end{equation}

By finding derivative of \eqref{p6} with respect to $\lambda$, we
reach to (the kinetic formulation of \eqref{scldf}; see \cite{PC})

\begin{align*}
%\label{p7}
&\dv \Bigl(h_n(\mx,\lambda)\partial_\lambda\vf(\mx,\lambda)\Bigr) -
\dv \dv  (h_n(\mx,\lambda) \partial_\lambda\mB(\mx, \lambda)) = -\pa_\lambda
G_n(\mx,\lambda)
\end{align*} where $h_n(\mx,\lambda)={\rm sgn}(u_n(\mx)-\lambda)$, and this is the special case of equation
\eqref{main-sys}. From here, we see that, due to Remark 16, the convergence
\eqref{uvjet51}, and the genuine nonlinearity conditions
\eqref{pgnl}, the sequence $(\varphi h_n)$ satisfies conditions of
Theorem \ref{main-result} (see also Remark 15). Thus it
follows that $(\int_{-M}^{M} h_n(\mx,\lambda)d\lambda)$ is
strongly precompact in $\Ldl\Omega$.
Since
$$
\int_{-M}^{M} h_n(\mx,\lambda)d\lambda=2
u_n(\mx),
$$ we conclude that $(u_n)$ is strongly $\Ldl\Omega$ precompact
itself.
\end{proof}

{\bf Acknowledgement} Darko Mitrovi\'c is engaged as a part time
researcher at the University of Bergen in the frame of the project
"Mathematical and Numerical Modeling over Multiple Scales" of the
Research Council of Norway whose support we gratefully acknowledge.

The work presented in this paper was also supported in part by the Ministry of Science, Education and Sports of the Republic of Croatia (project 037-0372787-2795), as well as by the DAAD project {\it Center of Excellence for Applications of Mathematics}.

\end{document}